\documentclass[journal]{IEEEtran}
\IEEEoverridecommandlockouts                              
\bstctlcite{IEEEexample:BSTcontrol}                                                          
\usepackage{flushend} 
\usepackage{romannum}
\usepackage[usenames,dvipsnames]{xcolor}
\usepackage{tkz-berge}
\usetikzlibrary{fit,shapes}                                     
\usepackage{calrsfs}
\DeclareMathAlphabet{\pazocal}{OMS}{zplm}{m}{n}
\usepackage{graphicx}
\usepackage{graphics} 
\usepackage{epsfig} 
\usepackage{mathptmx}
\usepackage{caption}
\usepackage{subcaption}
\usepackage{times} 
\usepackage{tikz}
\usetikzlibrary{positioning}
\usepackage{amsmath, amssymb}

\usepackage[perpage]{footmisc}
\usepackage{amsthm}
\usepackage{amsfonts} 
\usepackage{setspace}
\usepackage{multirow}
\usepackage{textcomp}
\usepackage[english]{babel}
\usepackage{float} 
\usepackage{algorithmicx}
\usepackage[section]{algorithm}
\usepackage{algpascal}
\usepackage{algc}
\usepackage{algcompatible}
\usepackage{algpseudocode}
\let\oldReturn\Return
\renewcommand{\Return}{\State\oldReturn}
\usepackage{mathtools}
\usepackage{bm}
\usepackage{mathptmx}
\usepackage{times} 
\usepackage{mathtools, cuted}
\usepackage{textcomp}
\usepackage[english]{babel}
\usepackage{rotating}
\usepackage{pgfplots}
\pgfplotsset{compat=1.5}
\usepackage{siunitx}
\usepackage{pgfplotstable}
\usepackage{filecontents}

\newcommand{\overbar}[1]{\mkern 1.5mu\overline{\mkern-1.5mu#1\mkern-1.5mu}\mkern 1.5mu}
\newcommand{\norm[1]}{\left\lVert#1\right\rVert}

\newcommand{\B}{\pazocal{B}}

\newcommand{\K}{\pazocal{K}}

\newcommand{\D}{\pazocal{D}}
\newcommand{\E}{\pazocal{E}}

\newcommand{\I}{\pazocal{I}}

\newcommand{\nT}{n_{\scriptscriptstyle T}}

\newcommand{\x}{{X}}
\renewcommand{\u}{{U}}
\newcommand{\dx}{{X'}}
\newcommand{\remove}[1]{}

\def \cN{{\mathcal N}}

\def \*{\star}
\def \10n{\!\!\!\!\!\!\!\!\!\!}

\def \betn {\beta_{\tiny{\I}}}
\def \alpi {\alpha_{\tiny{\cN}}}

\newcommand{\R}{\mathbb{R}}
\newcommand{\bA}{\bar{A}}

\newcommand{\bB}{\bar{B}}

\newcommand{\bE}{\bar{E}}

\newtheorem{theorem}{Theorem}
\newtheorem{lem}[theorem]{Lemma}

\newtheorem{defn}[theorem]{Definition}

\newtheorem{rem}[theorem]{Remark}
\newtheorem{prob}[theorem]{Problem}

\newtheorem{prop}[theorem]{Proposition}
\numberwithin{theorem}{section}

\graphicspath{{Figures/}}

\title{\LARGE \bf Optimal Selection of Interconnections in Composite Systems for Structural Controllability}
\author{Shana~Moothedath,
        Prasanna~Chaporkar
        and~Madhu~N.~Belur
\thanks{The authors are in the Department of Electrical Engineering, Indian Institute of Technology Bombay, India. Email: $\lbrace$shana, chaporkar, belur$\rbrace$@ee.iitb.ac.in.}}
\begin{document}
\maketitle
\thispagestyle{empty}
\pagestyle{empty}
\begin{abstract}
In this paper, we study structural controllability of a linear time invariant (LTI) {\em composite} system consisting of several subsystems. We assume that the neighbourhood of each subsystem is unconstrained, i.e., any subsystem can interact with any other subsystem. The interaction links between subsystems are referred as {\em interconnections}. 
We assume the composite system to be structurally controllable if all possible interconnections are present, and our objective is to identify the minimum set of interconnections required to keep the system structurally controllable. 
We consider structurally identical subsystems, i.e., the zero/non-zero pattern of the state matrices of the subsystems are the same, but dynamics can be different. We present a polynomial time optimal algorithm to identify the minimum cardinality set of interconnections that subsystems must establish to make the composite system structurally controllable. Our general result we apply to special cases, where the minimum number can be more directly obtained. We considered controller canonical form, so-called structurally cyclic systems and subsystems that are structurally controllable with  a single input. The minimum number of interconnections required depends on some indices we defined in the paper, {\em maximum commonality index}, $\alpi$, and {\em unique dilation index}, $\betn$. More connectedness of subsystems leads to lower total value of $\alpi+\betn$ and if the subsystems have fewer number of connected components $\alpi$ decreases. 
\end{abstract}
\section{Introduction}\label{sec:intro}
Recently there has been immense research advance in the area of large-scale dynamical systems collectively using concepts from control theory, network science and statistical physics \cite{LiuBar:16}. Very often these networks consists of interconnected smaller entities called {\it subsystems} interacting to its neighbours. We refer to the collective system as the {\it composite} system and the interaction links as {\it interconnections}. Our aim is to find a minimum cardinality set of interconnections that the subsystems should establish amongst others such that the composite system is controllable with a specified input matrix. 

Typically, complex networks are characterized by large system dimension. Hence, it is very important to device efficient frameworks to tackle various optimization problems on these systems.  Moreover, in most cases the system parameters of the complex networks are not known precisely because of uncertainties in the system model, time varying link weights of the graph and so on \cite{LiuBar:16}. Hence, for addressing system theoretic questions related to these networks, for instance controllability and feedback selection, researchers resort to the topological characteristics of the system. Control theoretic analysis of complex networks, when only the graph structure of the network is known, is done using `structural control' \cite{Lin:74}.

In this paper, we consider composite systems consisting of LTI {\em structurally identical} subsystems. Two subsystems are referred as structurally identical if the zero/non-zero pattern of their state matrices are the same. Thus, given a set of structurally identical subsystems, our objective is to devise a framework to find out the neighbours of each subsystem and the state informations to be passed that the composite system is controllable using minimum number of interconnections.
\vspace*{-2.63 mm}
\subsection{Problem Formulation}\label{sec:problem}
Consider a set of $k$ subsystems, $S_1, \ldots, S_k$, with state matrices $A_1, \ldots, A_k$ and an input matrix $B$.  Let the dynamics of the subsystems be given by
\begin{eqnarray}
\dot{x}_i(t)& = A_ix_i(t), \mbox{~for~} i=1,\ldots,k,
\end{eqnarray}
where $x_i \in \R^{n_i}$ denotes the state vectors and $A_i \in \R^{n_i \times n_i}$ denotes the state matrices. Here $\R$ denotes the set of real numbers. We define $E_{ij} \in \R^{n_i \times n_j}$ as the {\it connection matrix} from the $j^{\rm th}$ subsystem to the $i^{\rm th}$ subsystem. Composing the connection matrices, the {\it composite system} of $k$ subsystems is 
\begin{eqnarray}\label{eq:sys}
\dot{x}(t) &=& \underbrace{\small{\begin{bmatrix}
A_1 & E_{12} & \cdots & E_{1k}\\
E_{21} & A_2 & \cdots & E_{2k}\\
\vdots & \ddots & \ddots & \vdots\\
E_{k1} & E_{k2} & \cdots & A_k
\end{bmatrix}}}_A
x(t) + B u(t),
\end{eqnarray}
where $x = [x_1^{\scriptscriptstyle T}, \ldots, x_k^{\scriptscriptstyle T}]^{\scriptscriptstyle T} \in \R^{\nT}$ with $\nT = \sum_{i = 1}^k n_i$, $u \in \R^{m}$ and $B \in \R^{\nT \times m}$. Here, $x_i = [x_1^i,\ldots,x_{n_i}^i]^{\scriptscriptstyle T}$. A subsystem $S_i$, for $i \in \{1,\ldots, k\}$, is said to be an outgoing (incoming, resp.) neighbour of the subsystem $S_j$ if the connection matrix $E_{ij}$ ($E_{ji}$, resp.) is not the zero matrix. The set of incoming and outgoing neighbours are together referred as the {\it neighbours} of subsystem $S_j$.
We assume that for an arbitrary distinct ordered pair of subsystems $S_i, S_j$ all states of subsystem $S_j$ can connect to all states of subsystem $S_i$ and vice-versa. In other words, all entries of the matrices $E_{ij}$'s, for $i,j \in \{1,\ldots,k\}$, can possibly be non-zero. Our aim here is to establish the minimum number of interconnections such that the composite system given in equation~\eqref{eq:sys} is controllable using a specified input matrix. More precisely, we need to find sparsest $E_{ij}$'s for $i, j \in \{i,\ldots, k\}$ such that composite system $(A, B)$ is controllable. In this paper, we perform our analysis using structural systems theory. Now we define structured systems below.

The pair $(\bA, \bB)$ structurally represents a system $(A, B)$ if it satisfies the following:
\begin{eqnarray}\label{eq:struc}
\bA_{ij} &=& 0 \mbox{~whenever~} A_{ij} = 0,\mbox{~and} \nonumber \\
\bB_{ij} &=& 0 \mbox{~whenever~} B_{ij} = 0.
\end{eqnarray}
We refer to matrices $A$ and $B$ that satisfy \eqref{eq:struc} as a {\it numerical realization} of $\bA$ and $\bB$ respectively and $(\bA, \bB)$ as a {\it structured system}. Thus for a given $(A, B)$, $(\bA, \bB)$ structurally represent a class of control systems corresponding to all possible numerical realizations. The key idea in structural controllability is to determine controllability of the class of systems represented by $(\bA, \bB)$. Specifically, we have the following definition.

\begin{defn}\label{def:struccont}
The structured system $(\bA, \bB)$ is said to be structurally controllable if there exists at least one numerical realization $(A, B)$ such that $(A, B)$ is controllable.
\end{defn}

Even though the definition of structural controllability requires only one controllable realization, it is known that if a system is structurally controllable, then `almost all' numerical realizations of the same structure is controllable  \cite{Rei:88}. 
 For various applications in structural control see \cite{LiuSloBar:11}.

Now we formally define the optimization problem considered in this paper. We consider structurally identical subsystems, i.e, subsystems whose sparsity pattern are identical. Thus, $\bA_i = \bA_1$ for all $i \in \{1,\ldots, k\}$. We denote the structured state matrix of the subsystems as $\bA_s$.  Given a set of $k$ subsystems with identical structured state matrices $\bA_s \in \{0,\*\}^{n_s \times n_s}$ and a structured input matrix $\bB  \in \{0, \*\}^{\nT \times m}$, where $\nT = k \times n_s$, we want to find a set of sparsest connection matrices, $\bE_{ij} \in \{0,\*\}^{n_s \times n_s}$, for $i, j \in \{1,\ldots,k\}$, that the composite structured system (obtained corresponding to equation~\eqref{eq:sys}) along with the given $\bB$ is structurally controllable. Note that there are exponential number of ways that one can connect the subsystems. Our objective is to select the sparsest set of $\bE_{ij}$'s for $i,j = 1, \ldots, k$. Let $\K$ denote the set of all possible structured state matrices of the composite systems that can be formed using $k$ subsystems whose state matrices are $\bA_s$, such that the composite system is structurally controllable with the given $\bB$. Thus for all matrices in $\K$, the $(n_s \times n_s)$ diagonal blocks are $\bA_s$. Two matrices $\bA'$ and $\bA''$ in $\K$ differs only in the off-diagonal blocks. In other words, $\K$ consists of structured matrices $\bA' \in \{0,\*\}^{\nT \times \nT}$ that satisfies $\bA'_i = \bA_s$, for $i \in \{1,\ldots, k\}$, and $\bE'_{ij} \in \{0,\*\}^{n_s \times n_s}$, such that the resulting composite system $(\bA', \bB)$ is structurally controllable. Then, we need to solve the following optimization problem.
\begin{prob} \label{prob:similar_int}
Given $k$ subsystems with structured state matrix $\bA_s$ and structured input matrix $\bB$, find $\bA^\*~ \in~ \arg\min\limits_{\10n~~ \bA' \in \K}  \norm[\bA']_0.$ 
\end{prob}

Note that the set $\K$ is non-empty, since when all entries of $\bE_{ij}$'s are $\*$'s, for all $i,j \in \{1, \ldots, k\}$, the resulting composite system is structurally controllable. Solving the minimum interconnection problem is same as minimizing the non-zero entries in matrices in $\K$, since for all matrices in $\K$ the diagonal blocks are fixed and hence optimization is possible only corresponding to the off-diagonal blocks. This in turn is same as minimizing the interconnections. 
\subsection{Related Work, Motivation and Contributions}\label{sec:lit_sur}
\subsubsection*{Numerical Framework}
Controllability and observability of composite systems is addressed in \cite{CheDes:67}, \cite{DavWan:75}, \cite{WolHwa:74}. Composite  systems consisting of subsystems with similar dynamics is studied in \cite{Lun:86}, \cite{SunElb:91}. Conditions to check various system theoretic properties of composite systems when all subsystems are identical is given in \cite{Lun:86}. Reference \cite{SunElb:91} deals with decentralized controller synthesis of composite systems with identical subsystem dynamics and symmetric interconnections. 

\subsubsection*{Structured Framework}
 Structural analysis of composite systems is studied in literature where various conditions for checking structural controllability of composite systems in terms of subsystems are given (see \cite{Dav:77},  \cite{AndHon:82}, \cite{RecPer:91}, \cite{YanZha:95}, \cite{CarPeqAguKarJoh:17} and references therein). The algorithm given in \cite{CarPeqAguKarJoh:17} accomplishes this using a distributed algorithm. Our goal is to find a minimum set of interconnections that each subsystem must establish amongst other subsystems that the composite system is controllable or observable. While finding a minimum set of interconnections, we also identify a neighbour set of each subsystem and state informations that has to be communicated that the composite system is controllable utilizing the least number of interconnections. Although the subsystems considered in this paper are structurally identical, the numerical entries need not be the same.

\subsubsection*{Motivation}
There exists practically important class of composite systems, including robotic swarms, power grids and biological systems, consisting of similar entities (subsystems) interacting with each other towards performing a desired task. Further, in most of the applications it is desired to achieve the intended performance by keeping the inter-subsystem interactions the least.  The key focus of this paper is to find a minimum cardinality set of interconnections among subsystems such that the composite system is controllable with a specified input matrix. To the best of our knowledge this problem is not addressed in the literature. We tackle the problem from a structural framework, where instead of the numerical matrices the structured matrices are used.

\noindent We summarize this paper's contributions below.\\
\noindent $\bullet$ Given a set of structured subsystems with identical structured state matrices and a single input, we find the optimal number of interconnections that the subsystems should establish amongst each other such that the composite system is structurally controllable (Theorem~\ref{th:opt_value}).\\
\noindent $\bullet$ We give an algorithm of polynomial complexity to find a set of minimum cardinality interconnection edges that solves Problem~\ref{prob:similar_int}  (Algorithm~\ref{alg:similar_inter} and Theorem~\ref{th:opt}).\\
\noindent $\bullet$ The results and algorithm presented in this paper apply to the multi-input case, where $\bB \in \{0,\*\}^{\nT \times m}$. This is discussed in Section~\ref{sec:extension}.

The organization of this paper is as follows: in Section~\ref{sec:prelim}, we give few graph theoretic preliminaries used in the sequel and some existing results. In Section~\ref{sec:res}, for structurally identical subsystems, we first find the minimum number of interconnections required to make a composite system structurally controllable. Then we give a polynomial time algorithm for solving Problem~\ref{prob:similar_int}. In Section~\ref{sec:illus}, we demonstrate our algorithm using an illustrative example and also discuss few special cases and the extension to the multi-input case. In Section~\ref{sec:conclu}, we give some concluding remarks.
\section{Review of Essential Graph Theoretic Results}\label{sec:prelim}
The key idea behind considering graphs for  analysing structured systems is because we can represent the influences of states and inputs on each state through a directed graph.
 In order to capture the interactions of states and inputs efficiently, we construct few digraphs corresponding to a structured system $(\bA, \bB)$ as described below.

Consider a structured  system $(\bA, \bB)$, where $\bA \in \{0,\*\}^{n \times n}$ and $\bB \in \{0, \*\}^{n \times m}$. Then the state digraph $\D(\bA) := (V_X, E_X)$, where $V_X= \{x_1,\ldots, x_n\}$ and $(x_i, x_j) \in E_X \Leftrightarrow \bA_{ji} = \*$. The system digraph $\D(\bA, \bB) := (V_X\cup V_U, E_X \cup E_U)$, where $V_U = \{u_1,\ldots, u_m\}$ and $(u_i, x_j) \in E_U \Leftrightarrow \bB_{ji} = \*$.  A state $x_j$ is said to be {\it accessible} if there exists a path from some input node $u_i$ to $x_j$.
Using the {\it strong connectedness} of the digraph $\D(\bA)$ one can check the accessibility of states $\{x_1,\ldots,x_n\}$. A digraph is said to be strongly connected if for each ordered pair of vertices $(v_i,v_k)$,
there exists a path from $v_i$ to $v_k$. A {\it strongly connected component} (SCC) is a maximal strongly connected subgraph of a digraph. Thus all states of a structured system are accessible if and only if all SCCs are accessible. We characterize the SCCs as per the following definition.

\begin{defn}\label{def:SCC}
In a digraph, an SCC $\hat{\cN}$ is said to be {\em non-top linked} if there are no directed edges from the nodes of other SCCs into any node in $\hat{\cN}$.
\end{defn}
Thus all states in a subsystem are accessible if and only if all the non-top linked SCCs are accessible.
While accessibility of all states is necessary for structural controllability, it is not sufficient. In addition to accessibility the system digraph should satisfy a {\it no-dilation} condition. Given a set of nodes, presence of a node set $S \subset V_{\x}$ such that its neighbourhood node set $T(S)$ (where node $x_i \in T(S)$, if there exists a directed edge from $x_i$ to a node in $S$), satisfying $|T(S)| < |S|$ is called as {\it dilation}. Note that, $S \subset V_{\x}$ but $T(S) \subset V_{\x} \cup V_{\u}$.
 Presence of dilations in $\D(\bA, \bB)$ can be easily checked using a matching condition on the system bipartite graph $\B(\bA, \bB)$ defined below.
 
  Given a bipartite graph $G_B = ((V, \widetilde{V}), E)$, where $V \cup \widetilde{V}$ denotes the set of nodes satisfying $V \cap \widetilde{V} = \phi$ and $E \subseteq V \times \widetilde{V}$ denote the set of undirected edges, a matching $M$ is a collection of edges $M \subseteq E$ such that no two edges in the collection share the same endpoint. That is, for any $(i, j)$ and $(u, v) \in M$, we have $i \neq u$ and $j \neq v$, where $i,u \in V$ and $j,v \in \widetilde{V}$. A matching $M$ is said to be a perfect matching of the bipartite graph $G_B$ if $|M| = {\rm min}(|V|, |\widetilde{V}|)$. Further, given $G_B$ and a cost function $c$ from the set $E$ to the set of non-negative real numbers $\R_+$, a minimum cost perfect matching is a perfect matching $M$ such that $\sum_{e \in M}c(e) \leqslant \sum_{e \in M'}c(e)$, where $M'$ is any perfect matching in $G_B$ \cite{Die:00}. There exists an equivalent matching condition on a bipartite graph denoted by $\B(\bA, \bB)$, for the no-dilation condition. The construction of $\B(\bA, \bB)$ is explained here in two stages. In the first stage, the state bipartite graph $\B(\bA) := ((V_{\dx}, V_{\x}), \E_{\x})$ is constructed, where $V_{\x}=\{x_1,x_2, \ldots, x_n \}$, $V_{\dx}=\{x'_1,x'_2, \ldots, x'_n \}$ and $(x'_{j}, x_i) \in \E_{\x} \Leftrightarrow (x_i, x_j) \in E_{\x}$. Subsequently, the system bipartite graph $\B(\bA, \bB) := ((V_{\dx}, V_{\x} \cup V_{\u}), \E_{\x} 
 \cup \E_{\u})$ is constructed, where $V_{\u} = \{u_1,u_2, \ldots, u_m \}$ and $( x'_{j}, u_i) \in \E_{\u} \Leftrightarrow (u_i, x_j) \in E_{\u}$. The following results relates $\B(\bA, \bB)$ and the no-dilation condition.
\begin{prop}\cite[Theorem 2]{Ols:15}\label{prop:dil}
 A digraph $\D(\bA,\bB)$ has no dilation if and only if the bipartite graph $\B(\bA,\bB)$ has a perfect matching.
\end{prop}
 
Using the state accessibility condition and the no-dilation condition, Lin proved the following result for structural controllability. 

\begin{prop}\cite [pp.207]{Lin:74}\label{prop:lin}
The structured system $(\bA,\bB)$ is structurally controllable if and only if the associated digraph $\D(\bA,\bB)$ has no inaccessible states and has no dilations.
\end{prop}

Alternatively, a structured system is said to be controllable if and only if all non-top lined SCCs are accessible by some input and there exists a perfect matching in  $\B(\bA, \bB)$.
\section{Algorithm and Results}\label{sec:res}
Given a set of structured subsystems and a structured input matrix, we first find the minimum number of interconnections required to make the composite system structurally controllable. Subsequently, we propose a polynomial time algorithm to identify an optimum set of interconnections. This algorithm thus solves Problem~\ref{prob:similar_int}. Before explaining the algorithm, we first give few constructions and supporting results. 

Given $k$ subsystems with identical structured state matrices $\bA_s$ and an input matrix $\bB$, we first construct the composite system $(\bA, \bB)$ as follows: for each subsystem,  $i=1,\ldots, k$, the state digraph $\D_i(\bA_s) := (V_{X_i}, E_{X_i})$, where $V_{X_i} = \{x_1^i, \ldots, x_{n_s}^i\}$ and $(x_k^i, x_j^i) \in E_{X_i}$ if $\bA_{s_{jk}} = \*$. The state digraph of the composite system, denoted by $\D(\bA)$, is obtained by compounding $\D_i(\bA_s)$, for $i=1,\ldots, k$, with all possible interconnections. The set of interconnections is denoted by $E_\I$.  Note that, we assume all possible interconnections are feasible. In other words, for distinct $i,j \in \{1,\ldots, k\}$, $(x_p^i, x_q^j) \in E_\I$, for all $p,q \in \{1,\ldots,n_s\}$. Then, $\D(\bA):= (\cup_{i=1}^kV_{X_i}, \cup_{i=1}^k E_{X_i} \cup E_\I)$.  Here, $\bA \in \{0, \*\}^{\nT \times \nT}$, where $\nT = k\times n_s$. Now we construct the composite system digraph $\D(\bA, \bB) := (\cup_{i=1}^kV_{X_i}\cup V_U, \cup_{i=1}^k E_{X_i} \cup E_\I \cup E_U)$. For $\bB \in \{0,\*\}^{\nT \times m}$, $V_U = \{u_1,\ldots, u_m\}$ and  $(u_i, x_p^j) \in E_U$ if the matrix $\bB$ has $\*$ at the $(n_s(j-1)+p)^{\rm th}$-row and $i^{\rm th}$-column position. This completes the construction of the digraphs associated with the composite system.

Now we will discuss the construction of the bipartite graphs associated with the composite system. The state bipartite graph $\B(\bA) := ((\cup_{i=1}^k V_{X'_i}, \cup_{i=1}^kV_{X_i}),\cup_{i=1}^k\E_{X_i} \cup \E_\I)$, where $V_{X'_i} = \{{x'^i_1}, \ldots, {x'^i_{n_s}}\}$, $({x'}_p^i, x_q^i) \in \E_{X_i} \Leftrightarrow (x_q^i, x_p^i) \in E_{X_i}$ and $({x'}_p^i, x_q^j) \in \E_\I \Leftrightarrow (x_q^j, x_p^i) \in E_\I$. Further, the system bipartite graph of the composite system $\B(\bA, \bB) := ((\cup_{i=1}^k V_{X'_i}, \cup_{i=1}^kV_{X_i}\cup V_U),\cup_{i=1}^k\E_{X_i}\cup \E_\I \cup \E_U)$, where $({x'}_p^i, u_j) \in \E_U \Leftrightarrow (u_j, x_p^i) \in E_U$. 
Let $\cN= \{\cN_1,\ldots, \cN_q\}$ denote the set of non-top linked SCCs of $\D_i(\bA_s)$, for $i \in \{1,\ldots, k\}$, that are not accessible in the digraph with vertex set $\cup_{i=1}^kV_{X_i}\cup V_U$ and edge set $\cup_{i=1}^k E_{X_i} \cup E_U$. In other words, these are the non-top linked SCCs of $\D_i(\bA_s)$, for $i \in\{1,\ldots,k\}$, that are not accessible when interconnections are not present. Thus these non-top linked SCCs can be made accessible only using interconnections.  With some abuse of notation we denote the condensed
version of $\cN_1,\ldots, \cN_q$ using the same notation. We now define a bipartite graph $\B(\bA, \bB, \cN) := ((\cup_{i=1}^k V_{X'_i}, \cup_{i=1}^kV_{X_i}\cup V_U\cup \cN),\cup_{i=1}^k\E_{X_i}\cup \E_\I \cup \E_U \cup \E_\cN)$, where $(x'^i_p, \cN_j) \in \E_\cN \Leftrightarrow x^i_p \in \cN_j$. Further, define the cost function $c$ as given below.
\begin{equation}
c(e) \leftarrow \begin{cases}
	0, \mbox{~for~} e \in \E_U,\\
	1, \mbox{~for~} e \in \cup_{i=1}^k\E_{X_i},\\  
	2, \mbox{~for~} e \in \E_\cN, \label{eq:cost}\\        
    3, \mbox{~for~} e \in \E_\I.
  \end{cases}  
\end{equation}
The discussions in this paper is for the single input case, i.e., $\bB = \{0,\*\}^{\nT \times 1}$. Thus $V_U = u_1$. With respect to $\B(\bA, \bB)$, we have the following results.

\begin{lem}\label{lem:input_node_1}
Let $\bA$ be the structured state matrix obtained by composing $k$ subsystems with all possible interconnections  and let $\bB$ be a single input matrix. Let $\overbar{M}$ be an optimum matching obtained by solving the minimum cost perfect matching on the bipartite graph $\B(\bA, \bB)$ under cost function $c$ given in equation~\eqref{eq:cost}. Then, $({x'_i}^j, u_1) \in \overbar{M}$ for some $i \in \{1,\ldots,n_s\}$ and $j \in \{1,\ldots, k\}$.
\end{lem}
\begin{proof}
Given $\overbar{M}$ is an optimum matching of $\B(\bA, \bB)$. We prove the result using a contradiction argument. Suppose $({x'_i}^j, u_1) \notin \overbar{M}$ for all $i \in \{1,\ldots,n_s\}$ and for all $j \in \{1,\ldots,k\}$. Then, since $\overbar{M}$ is a perfect matching, $({x'}_i^j, x_r^t) \in \overbar{M}$ for some node $x_r^t$. Construct a new matching $M'$ by breaking the edge $({x'}_i^j, x_r^t)$ and making the edge $({x'}_i^j, u_1)$, i.e., $M' = \{ \overbar{M} \setminus ({x'}_i^j, x_r^t)\} \cup \{({x'}_i^j, u_1)\}$. Notice that $c(M') < c(\overbar{M})$. This contradicts the assumption that $\overbar{M}$ is an optimum matching in $\B(\bA, \bB)$ and thus $({x'}_i^j, u_1) \in \overbar{M}$. 
\end{proof}

\begin{lem}\label{lem:unmatched_accessible}
Consider an optimum perfect matching $M$  in $\B(\bA, \bB)$ and let $\E_{\I'} = M \cap \E_\I$. Then, there exists a right unmatched node in $M$ such that it is accessible in the digraph constructed with vertex set $\cup_{i=1}^k V_{X_i} \cup V_U$ and edge set $\cup_{i=1}^k E_{X_i} \cup E_U \cup E_{\I'}$, where $(x^i_p,x^j_q) \in E_{\I'} \Leftrightarrow (x'^j_q, x^i_p) \in \E_{\I'}$.
\end{lem} 
\begin{proof}
The bipartite graph $\B(\bA, \bB)$ consists of $\nT$ left side nodes and $\nT+1$ right side nodes, where one extra node in the right side is the input node $u_1$. Thus any perfect matching in $\B(\bA, \bB)$ has size $\nT$ and hence in $M$ there is one right unmatched node. Now we need to show that this unmatched node is accessible in the digraph constructed using vertex set $\cup_{i=1}^k V_{X_i} \cup V_U$ and edge set $\cup_{i=1}^k E_{X_i} \cup E_U \cup E_{\I'}$, where $(x^i_p,x^j_q) \in E_{\I'} \Leftrightarrow (x'^j_q, x^i_p) \in M \cap \E_{\I}$. By Lemma~\ref{lem:input_node_1}, all optimum perfect matchings in $\B(\bA, \bB)$ consists of an edge $(x'^i_p, u_1)$ for some node $x'^i_p$. Let $(x'^i_p, u_1) \in M$. Thus the node $x^i_p$ is accessible in the specified digraph. Now in the matching $M$, the node $x^i_p$ satisfies one of the following: (a)~$x^i_p$ is unmatched, or (b)~$x^i_p$ is matched. In case~(a) the proof follows. In case~(b), let $(x'^j_q, x^i_p) \in M$. Then the node $x^j_q$ is accessible. Recursively using the same argument as before, we can say that the unmatched node in $M$ is accessible in the digraph constructed using vertex set $\cup_{i=1}^k V_{X_i} \cup V_U$ and edge set $\cup_{i=1}^k E_{X_i} \cup E_U \cup E_{\I'}$.
\end{proof}

Lemma~\ref{lem:unmatched_accessible} concludes that with respect to any optimum perfect matching in $\B(\bA, \bB)$, there exists a unique right unmatched accessible node in $\B(\bA, \bB)$. Let  $M^\*$ be an optimum matching in $\B(\bA, \bB, \cN)$  under cost function $c$. Now we give the following preliminary result to show that the input node $u_1$ is selected in $M^\*$.

\begin{lem}\label{lem:input_node}
 Let $M^\*$ be an optimum matching obtained by solving the minimum cost perfect matching on the bipartite graph $\B(\bA, \bB, \cN)$ under cost function $c$ given in equation~\eqref{eq:cost}. Then, $({x'_i}^j, u_1) \in M^\*$ for some $i \in \{1,\ldots,n_s\}$ and $j \in \{1,\ldots, k\}$.
\end{lem}
\begin{proof}
$\B(\bA, \bB, \cN)$ is a bipartite graph with $\nT$ vertices on the left side and $\nT+1+q$ vertices on the right side. Given $M^\*$ is an optimum matching of $\B(\bA, \bB, \cN)$. We prove the result using a contradiction argument. Suppose $({x'_i}^j, u_1) \notin M^\*$ for all $i \in \{1,\ldots,n_s\}$ and for all $j \in \{1,\ldots,k\}$. Then, since $M^\*$ is a perfect matching in $\B(\bA, \bB, \cN)$,  $({x'}_i^j, v) \in M^\*$ for some node $v$. Construct a new matching $M'$ by breaking the edge $({x'}_i^j, v)$ and making the edge $({x'}_i^j, u_1)$, i.e., $M' = \{M^\* \setminus ({x'}_i^j, v)\} \cup \{({x'}_i^j, u_1)\}$. Notice that $c(M') < c(M^\*)$. This contradicts the assumption that $M^\*$ is an optimum matching in $\B(\bA, \bB, \cN)$ and thus the proof follows. 
\end{proof}
Thus by Lemma~\ref{lem:input_node}, $|M^\* \cap \E_U| = 1$. Henceforth in the sequel $M^\*$ denotes an optimum matching in $\B(\bA, \bB, \cN)$, such that $|M^\* \cap \E_\cN| = \alpi$ and $|M^\* \cap \E_\I| = \betn$. Thus $|M^\* \cap\cup_{i=1}^k \E_{X_i}| = \nT-(1+\alpi+\betn)$. With respect to $M^\*$, we now prove the following lemmas. 

\begin{lem}\label{lem:alpha+beta_matching}
 Let $M^\*$ be an optimum matching in the bipartite graph $\B(\bA, \bB, \cN)$ under cost function $c$. Let $M^\*$ satisfy $|M^\* \cap \E_\cN| = \alpi$ and $|M^\* \cap \E_\I| = \betn$. Then, any minimum cost perfect matching $\widetilde{M}$ in $\B(\bA, \bB)$ satisfies $|\widetilde{M} \cap \E_\I| = \alpi + \betn$.
\end{lem}
\begin{proof}
Given $M^\*$ is an optimum matching in $\B(\bA, \bB, \cN)$. We first prove the existence of a perfect matching $M$ in $\B(\bA, \bB)$ satisfying $|M \cap \E_\I| = \alpi + \betn$. For this we construct a matching $M$ from $M^\*$ such that $|M \cap \E_\I| = \alpi + \betn$. Given $M^\*$ satisfies $|M^\* \cap \E_\cN| = \alpi$ and $|M^\* \cap \E_\I| = \betn$. By Lemma~\ref{lem:input_node}, $|M^\* \cap \E_U| = 1$. Thus, $|M^\* \cap \cup_{i=1}^k\E_{X'_i}| = \nT-(\alpi + \betn +1)$. Let $M' \subset M^\*$ is defined as $M' = M^\* \cap \{\cup_{i=1}^k\E_{X'_i} \cup \E_U\} $. Thus $|M'| = \nT-\alpi-\betn$. Note that $M'$ is a matching in $\B(\bA, \bB)$. Consider the bipartite graph $\B(\bA, \bB) \odot M'$, where $\odot$ denotes a difference operation in which all nodes with non-zero degree in $M'$ and the edges associated with these nodes are removed from $\B(\bA, \bB)$. More precisely, $\B(\bA, \bB) \odot M'$ consists of only those nodes in $\B(\bA, \bB)$ that are not matched in $M'$ and the edges between those nodes in $\B(\bA, \bB)$. Thus $\B(\bA, \bB) \odot M'$ is a bipartite graph with $\alpi + \betn$ nodes on the left side and $\alpi +\betn +1$ nodes on the right side.  Notice that $\B(\bA, \bB) \odot M'$ has a perfect matching. This is because since all subsystems are structurally identical and $M' = M^\* \cap \{\cup_{i=1}^k\E_{X'_i} \cup \E_U\} $, where $M^\*$ is an optimum matching in $\B(\bA, \bB, \cN)$ under cost function $c$, the number of nodes in $\B(\bA, \bB) \odot M'$ corresponding to each subsystem is the same in both left and right sides except for one subsystem. For one subsystem  (the $i^{\rm th}$ subsystem if $(x'^i_p, u_1) \in M'$ for some $p \in \{1,\ldots, n_s\}$) either the number of nodes in the left side of $\B(\bA, \bB)\odot M'$ is one \underline{less} than other subsystems or the number of nodes in the right side is one \underline{more} than the other subsystems. In both cases there exists a perfect matching in $\B(\bA, \bB) \odot M'$. Let $M''$ be an optimum perfect matching in $\B(\bA, \bB) \odot M'$ using cost function $c$. Then $M'' \subset \E_\I$. This is because if an edge in $\cup_{i=1}^k \E_{X_i} \cup \E_U$ is present in $M''$, then it contradicts the optimality of $M^\*$. Now $\widetilde{M} = M' \cup M''$ is a perfect matching in $\B(\bA, \bB)$. Note that,  $|\widetilde{M} \cap \E_\I| = \alpi + \betn$. This proves that there exists a perfect matching in $\B(\bA, \bB)$ consisting of $\alpi + \betn$ interconnections. Since $\widetilde{M}$ is constructed from $M^\*$ which is an optimum matching under cost function $c$, the optimality of $\widetilde{M}$ follows. 
\end{proof}
Thus by Lemma~\ref{lem:alpha+beta_matching}, we conclude that optimum matching in $\B(\bA, \bB)$ under cost $c$ has $\alpi + \betn$ interconnections. An intuitive explanation of these indices, $\alpi$ and $\betn$, is given in Remark~\ref{rem:constants}. In the result below we prove that from $\alpi + \betn$ interconnections in $\widetilde{M}$, the $\alpi$ interconnections connects to states in $\alpi$ distinct SCCs in the set $\cN = \{ \cN_1, \ldots, \cN_q\}$.

\begin{lem}\label{lem:alpha+beta_scc}
Let $\widetilde{M}$ be an optimum matching in $\B(\bA, \bB)$ obtained from an optimum  matching $M^\*$ in $\B(\bA, \bB, \cN)$ such that $|M^\* \cap \E_\cN| = \alpi$ and $|M^\* \cap \E_\I| = \betn$. Then, there exists $\{e_1, \ldots, e_{\alpi}\} \in \widetilde{M} \cap \E_\I$ such that for $e_t = (x'^i_p, x^j_q)$, $x^i_p \in \cN_t$, where $\cN_t \in \cN$. 
\end{lem}
\begin{proof}
Recall the construction of the matching $\widetilde{M}$ given in the proof of Lemma~\ref{lem:alpha+beta_matching}. In $\widetilde{M}$ there are exactly $\alpi + \betn
$ interconnection edges. Out of these interconnections $\alpi$ interconnections has left side nodes from distinct SCCs, say $\cN_1, \ldots, \cN_{\alpi}$. Thus the proof of the claim follows.
\end{proof}

In the result below we prove the existence of an optimum perfect matching $\hat{M}$ in $\B(\bA, \bB)$ that ensures accessibility of $\alpi$ SCCs in $\cN$ using only the interconnections in $\hat{M}$. 
\begin{lem}\label{lem:alpha_access}
Let $\bA$ be the structured state matrix obtained by composing $k$ subsystems with all possible interconnections  and let $\bB$ be an input matrix. Then, there exists an optimum matching $\hat{M}$ in $\B(\bA, \bB)$ such that $|\hat{M} \cap \E_\I| = \alpi + \betn$. Further, SCCs $\{ \cN_1, \ldots, \cN_{\alpi} \} \in \cN$ are accessible in the digraph with vertex set $\cup_{i=1}^k V_{X_i} \cup V_U$ and edge set $\cup_{i=1}^k E_{X_i} \cup E_U \cup E_{\I'}$, where $(x^j_q, x^i_p) \in E_{\I'} \Leftrightarrow (x'^i_p,x^j_q) \in \hat{M} \cap \E_{\I}$.
\end{lem}
\begin{proof}
We  know, from Lemma~\ref{lem:alpha+beta_matching}, that there exists an optimum matching $\widetilde{M}$ in $\B(\bA, \bB)$ such that $|\widetilde{M} \cap \E_{\I}| = \alpi + \betn$. Also, from Lemma~\ref{lem:alpha+beta_scc} at least $\alpi$ left side nodes in $\widetilde{M} \cap \E_\I$ are from $\alpi$ distinct SCCs. Let $x'^i_p$ be an arbitrary node such that $(x'^i_p, x^j_q) \in \widetilde{M} \cap \E_{\I}$, $x^i_p \in \hat{\cN}$, $\hat{\cN} \in \cN$ and $\hat{\cN}$ is inaccessible in the digraph with vertex set $\cup_{i=1}^k V_{X_i} \cup V_U$ and edge set $\cup_{i=1}^k E_{X_i} \cup E_U \cup E_{\I'}$. By Lemma~\ref{lem:unmatched_accessible} we know that there exists a unique unmatched node in $\widetilde{M}$ that is accessible, say $\hat{x}$. Then $\hat{x}$ satisfies one of the following cases: (a)~$\hat{x}$ is in the same subsystem as of $\hat{\cN}$, or (b)~$\hat{x}$ is not in the subsystem of $\hat{\cN}$. We will resolve case~(b) first. Construct a new matching $\hat{M}$ such that $\hat{M} = \{\widetilde{M} \setminus (x'^i_p, x^j_q) \}\cup \{ (x'^i_p, \hat{x})\}$. Note that in $\hat{M}$ the number of interconnections is the same as in $\widetilde{M}$ and further the SCC $\hat{\cN}$ is accessible. Now we will resolve case~(a). In case~(a), note that $\hat{x}$ is in the same subsystem as $\hat{\cN}$. Since the unique unmatched node $\hat{x}$ is in the $i^{\rm th}$ subsystem and $(x'^i_p, x^j_q) \in \widetilde{M}$, there exists an interconnection edge in $\widetilde{M}$ matching a left side node in $j^{\rm th}$ subsystem to some node in a different subsystem, say $(x'^j_r, x^v_w) \in \widetilde{M}$, $j \neq v$. Construct a new matching $\hat{M} = \{\widetilde{M} \setminus (x'^j_r, x^v_w) \} \cup \{ (x'^j_r, \hat{x}) \}$. This is possible since $i \neq j$. Notice that $|\hat{M}\cap \E_\I| = \alpi+\betn$ and SCC $\hat{\cN}$ is accessible. Since $x'^i_p$ is arbitrary the proof follows.
\end{proof}
Using Lemmas~\ref{lem:alpha+beta_matching} and~\ref{lem:alpha_access}, now we give one of the main result of this paper.

\begin{theorem}\label{th:opt_value}
Let $\bA^\*$ be an optimum solution to Problem~\ref{prob:similar_int} and let $|\I^\*|$ be the number of interconnections in $\bA^\*$. Further, let  $q = |\cN|$ and let $M^\*$ be an optimum matching in the bipartite graph $\B(\bA, \bB, \cN)$ such that $|M^\* \cap \E_\cN| = \alpi$ and $|M^\* \cap \E_\I| = \betn$. Then, $|\I^\*| = \betn + q$.
\end{theorem}
\begin{proof}
We prove this result in two steps. In step~(i) we show that $|\I^\*| \leqslant \betn + q$ and in step~(ii) we show that $|\I^\*| \geqslant \betn + q$. Thus combining Steps~(i) and~(ii) the result follows.  

\noindent{\bf Step~(i)}: Here we will prove that $|\I^\*| \leqslant \betn + q $.
From Lemmas~\ref{lem:alpha+beta_matching} and~\ref{lem:alpha_access}, we know that there exists a perfect matching $\hat{M}$ in $\B(\bA, \bB)$ that uses exactly $\alpi + \betn$ interconnections and out of $q$ SCCs in $\cN$, $\alpi$ SCCs are accessible using these interconnections. Thus the number of SCCs in $\cN$ that are not accessible after using the interconnections in $\hat{M}$ is $q - \alpi$. Accessibility of these SCCs can be achieved by adding $q-\alpi$ interconnections more. Thus using $(\alpi+ \betn)+(q-\alpi) = \betn + q$ interconnections, all SCCs are accessible and there exists a perfect matching in $\B(\bA, \bB)$. Thus we can compose the subsystems using $\betn+q$ interconnections such that the composite system is structurally controllable. Hence $|\I^\*| \leqslant \betn + q$.

\noindent{\bf Step~(ii)}: Here we will prove that $|\I^\*| \geqslant q + \betn$.
We prove this using a contradiction argument. Suppose not. Then $|\I^\*| < q + \betn$. This implies $|\I^\*| \leqslant q + \betn - 1$. With out loss of generality, assume that $|\I^\*| = q + \betn - 1$. Thus we can compose the subsystems using $q + \betn - 1$ interconnections such that the composite system is structurally controllable. Consider an optimum matching $\overbar{M}$ in $\B(\bA, \bB)$ under cost function $c$. We know from Lemma~\ref{lem:alpha+beta_matching} that $\overbar{M}$ consists of $\alpi+\betn$ interconnections. Thus $[(q + \betn - 1) - (\alpi+\betn)]= q-\alpi-1$ interconnections are solely for achieving accessibility condition. This implies that $\alpi+1$ SCCs are accessible using the interconnections in $\overbar{M}$. Note that SCCs, $\cN_1, \ldots, \cN_q$, are those SCCs whose states do not have a directed path from the input node when interconnections are not used. Hence at least one node in each of the $\alpi+1$ SCCs are connected using interconnection edges in $\overbar{M}$. Now we will construct a matching in $\B(\bA, \bB, \cN)$ from $\overbar{M}$. Note that $\overbar{M}$ is a perfect matching in $\B(\bA, \bB, \cN)$ also. Let $\E_q$ denotes the set of interconnections connecting one node each of $\alpi+1$ SCCs in $\overbar{M}$. Then, $|\E_q| = \alpi+1$. Remove $\E_q$ edges from $\overbar{M}$ and connect them to $\alpi+1$ SCC nodes, say $\{\cN_1, \ldots, \cN_{\alpi+1}\}$, in the right side of $\B(\bA, \bB, \cN)$. Let this new set of edges is denoted by $\E_{\alpi+1}$. Then, $M'' = \{\overbar{M} \setminus \E_q\}\cup \{ \E_{\alpi+1} \}$. The cost of this new matching is $3[(\alpi+\betn)-(\alpi+1)]+2(\alpi+1) = 3\,\betn+2\,\alpi-1$. Note that cost of optimum matching $M^\*$ in $\B(\bA, \bB, \cN)$ is $3\,\betn + 2\,\alpi$. Thus $c(M'') < c(M^\*)$. This contradicts that $M^\*$ is an optimum matching in $\B(\bA, \bB, \cN)$. Hence the assumption $|\I^\*| < q + \betn$ is not true. Thus $|\I^\*| \geqslant q + \betn$. Thus from Steps~(i) and~(ii), $|\I^\*| = q+\betn$.
This completes the proof.
\end{proof}

\begin{algorithm}[t]
  \caption{Pseudo-code for solving Problem~\ref{prob:similar_int} on structured subsystems}\label{alg:similar_inter}
  \begin{algorithmic}
\State \textit {\bf Input:} $k$ structured subsystems with state matrices $\bA_s$ and a single input matrix $\bB$
\State \textit{\bf Output:} Interconnections $\I_A$
\end{algorithmic}
  \begin{algorithmic}[1]
      \State Construct the bipartite graph $\B(\bA, \bB)\leftarrow ((\cup_{i=1}^k V_{X'_i}, \cup_{i=1}^k V_{X_i}\cup V_U), \cup_{i=1}^k \E_{X_i}\cup \E_\I \cup \E_U)$ \label{step:bipartite1}
      \State Construct the bipartite graph $\B(\bA, \bB, \cN)\leftarrow ((\cup_{i=1}^k V_{X'_i}, \cup_{i=1}^k V_{X_i}\cup V_U\cup \cN), \cup_{i=1}^k \E_{X_i}\cup \E_\I \cup \E_U \cup \E_\cN)$ \label{step:bipartite2}
      \State Define cost vector $c$ as in equation~\eqref{eq:cost}
    \State Find min cost max matching in $\B(\bA, \bB, \cN)$, say $M^\*$ \label{step:MCMM}
    \State $M' \leftarrow M^\* \cap \{\cup_{i=1}^k\E_{X'_i} \cup \E_U\}$\label{step:M'}
    \State Find a perfect matching in $\B(\bA, \bB) \odot M'$, say $M''$\label{step:pm}
    \State $\widetilde{M} \leftarrow M' \cup M''$\label{step:Mt}
    \State $V_{\cN'} \leftarrow \{x'^i_p: (x'^i_p, x^j_q) \in \widetilde{M}, x^i_p \in \cN_h, \cN_h \mbox{~not~accessible~} \}$\label{step:Vn}
    \State Let $x^t_r$ be the unique unmatched accessible node in $\widetilde{M}$\label{step:aceess_node}
    \While{$|V_{\cN'}| \neq 0$}
      \If{$x'^i_p \in V_{\cN'} \mbox{~and~} i \neq t$}
         \State $\widetilde{M} \leftarrow \{\widetilde{M} \setminus (x'^i_p, x^j_q) \} \cup \{(x'^i_p, x^t_r)\} $\label{step:Mt_1}
      \ElsIf {$x'^i_p \in V_{\cN'} \mbox{~and~} i = t$}
         \State Find $x^w_v$ such that for $s\in \{1,\ldots, n_s\},$ $(x'^j_s, x^w_v) \in \widetilde{M}$\label{step:find_breakable_edge}
         \State $\widetilde{M} \leftarrow \{\widetilde{M} \setminus (x'^j_s, x^w_v) \} \cup \{(x'^j_s, x^t_r)\} $\label{step:Mt_2}
      \EndIf
    \EndWhile\label{step:endwhile} 
    \State $\E^q_\I \leftarrow \{(x'^i_j,x^t_r): \mbox{~SCC~of~} i^{\rm th}$ subsystem~is not accessible in the digraph with vertex set $(\cup_{i=1}^k V_{X_i} \cup V_U)$ and edge set  $\cup_{i=1}^k \E_{X_i} \cup \E_U \cup (\widetilde{M}\cap \E_\I), x^t_r \mbox{~is~accessible~and~} i\neq t\}$\label{step:extra_accessibility}
 \State $\I_A \leftarrow \{(x_p^i, x_q^j):({x'}_q^j, x_p^i) \in \widetilde{M} \mbox{ and } i \neq j\} \cup \{{\E^q_\I} \}$
\label{step:soln}
\end{algorithmic}
\end{algorithm}

Now we give an optimal algorithm to solve Problem~\ref{prob:similar_int} in polynomial time.  
The pseudo-code of the proposed algorithm is given in Algorithm~\ref{alg:similar_inter}.

\noindent{\bf Steps~1-4}:  Initially we run a minimum cost perfect matching algorithm on the bipartite graph $\B(\bA, \bB, \cN)$ using cost function $c$. Let $M^\*$ be the optimum matching obtained. Let $|M^\* \cap \E_\I| = \betn$ and $|M^\* \cap \E_\cN| = \alpi$. From Lemma~\ref{lem:input_node}, we know that $|M^\* \cap \E_U| = 1$.

\noindent{\bf Steps~5-7}: Now we define a matching $M':= M^\* \cap \{\cup_{i=1}^k\E_{X'_i} \cup \E_U\}$. Note that $|M'| = \nT-(\alpi+\betn)$. Subsequently, we find the difference of $\B(\bA, \bB)$ and $M'$, denoted as $\B(\bA, \bB) \odot M'$. $\B(\bA, \bB) \odot M'$ consists of only those nodes in $\B(\bA, \bB)$ that are not matched in $M'$ and the edges between them.  Moreover, there exists a perfect matching $M''$ in $\B(\bA, \bB) \odot M'$ such that $M'' \subset \E_\I$ (see proof of Lemma~\ref{lem:alpha+beta_matching}). We define $\widetilde{M}$ as the union of $M'$ and $M''$. Note that, $\widetilde{M}$ is a perfect matching in $\B(\bA, \bB)$. Further, $|\widetilde{M} \cap \E_\I| = \alpi+\betn$ and $\alpi$ interconnections connects to states in $\alpi$ distinct SCCs. However, these $\alpi$ SCCs need not be accessible using these $\alpi+\betn$ interconnections in $\widetilde{M}$. Thus our aim is to redefine $\widetilde{M}$ in such a way that, in the new $\widetilde{M}$, $|\widetilde{M} \cap \E_\I| = \alpi+\betn$ and $\alpi$ SCCs are accessible using interconnections in $\widetilde{M} \cap \E_\I$.

\noindent{\bf Steps~8-17}: For achieving the accessibility of $\alpi$ SCCs, we first identify the $\alpi$ interconnection edges in $\widetilde{M}$ that connects to one state each in SCCs, $\cN_1,\ldots,\cN_{\alpi}$. Let $V_{\cN'}$ is the set of $\alpi$ left side nodes in $\B(\bA, \bB)$ belonging to SCCs that are matched through edges in $\E_\I$ in $\widetilde{M}$. Further these SCCs are not accessible even after using interconnections in $\widetilde{M}\cap \E_\I$. By Lemma~\ref{lem:unmatched_accessible}, we know that in $\widetilde{M}$ there exists an unmatched accessible node. Let $x^t_r$ be this node. Our idea is to break the edges corresponding to nodes in $V_{\cN'}$ from $\widetilde{M}$ and make new interconnections using the node $x^t_r$ such that SCCs become accessible. Consider an arbitrary vertex $x'^i_p \in V_{\cN'}$. Let $x^i_p \in \cN_h$. $x'^i_p$ satisfies one of the following cases: (a)~ $i \neq t$ or (b)~$i=t$. In case~(a), we redefine $\widetilde{M}$ by breaking the edge $(x'^i_p, x^j_q)$ and making the edge $(x'^i_p, x^t_r)$. Note that, in the updated $\widetilde{M}$, SCC $\cN_h$ is accessible. In case~(b), the unique unmatched accessible node belongs to the same subsystem as $x^i_p$. Thus edge $(x'^i_p, x^t_r)$ cannot be formed. However, notice that since the unique unmatched node is in $i^{\rm th}$ subsystem and  $(x'^i_p, x^j_q) \in \widetilde{M}$, there exists some edge $(x'^j_s, x^w_v)$ for some $s, v \in \{1,\ldots, n_s\}$ and $j \neq w$ in $\widetilde{M}$. Thus in case~(b), we redefine  $\widetilde{M}$ by breaking this edge and making the edge $(x'^j_s, x^t_r)$. Now SCC $\cN_h$ is accessible. This is achieved by keeping the number of interconnections the same as before, i.e., $\alpi+\betn$. Thus by the end of the Step~\ref{step:endwhile} $\alpi$ SCCs are accessible using $\widetilde{M}\cap \E_\I$ interconnections alone. Thus the number of not accessible SCCs are $q-\alpi$.

\noindent{\bf Steps~18-19}: Now we add $q-\alpi$ interconnections one each to some state in these $q-\alpi$ SCCs from accessible nodes in other subsystems. These set of edges that are added to attain accessibility of $q-\alpi$ SCCs are denoted by $\E^q_\I$. Thus using $[(\alpi+\betn) + (q-\alpi)] = q+\betn$ interconnections, we achieve accessibility of all SCCs $\{\cN_1, \ldots, \cN_q\}$ and a perfect matching in $\B(\bA, \bB)$. The final interconnection edge set is given by $\I_A$. This completes the description of Algorithm~\ref{alg:similar_inter}.

\begin{theorem}\label{th:opt}
Algorithm~\ref{alg:similar_inter}  which takes as input $k$ structured subsystems with state matrices $\bA_s$ of dimensions $n_s$ and input matrix $\bB$, gives as output the interconnection edges $\I_A$, which is an optimal solution to Problem~\ref{prob:similar_int}, i.e., $|\I_A| = |\I^\*|$.
\end{theorem}
\begin{proof}
By Theorem~\ref{th:opt_value} we know that the minimum number of interconnections that solve Problem~\ref{prob:similar_int} is $|\I^\*| = \betn + q$. Algorithm~\ref{alg:similar_inter} achieves accessibility of all states and no-dilation condition of the composed system using exactly $\betn+q$ interconnections. Thus, output of Algorithm~\ref{alg:similar_inter} is an optimal solution to Problem~\ref{prob:similar_int}.
\end{proof}

\begin{theorem}\label{th:comp1}
Algorithm~\ref{alg:similar_inter} which takes as input $k$ structured subsystems with state matrices $\bA_s$ of dimensions $n_s$, and input matrix $\bB$ and gives as output the interconnection edges $\I_A$ has running time complexity $O(\nT^3)$, where $\nT = k \times n_s$. 
\end{theorem}
\begin{proof}
Constructing the bipartite graphs $\B(\bA, \bB), \B(\bA, \bB, \cN)$ and solving the minimum cost perfect matching problem has complexity $O(\nT^3)$, where $\nT = k \times n_s$. The rest of the constructions are of linear complexity with maximum $\nT$ iterations. Thus complexity of Algorithm~\ref{alg:similar_inter} is $O(\nT^3)$.
\end{proof}
\begin{rem}\label{rem:constants}
Let $\Gamma_A$ be the minimum cardinality subset of all interconnections which can be used to achieve accessibility. Thus, $|\Gamma_A| = q$ (since there are $q$ non-top linked SCCs that are not accessible from the input). Also, let $\Gamma_D$ be the minimum set of interconnections which can be used to achieve the no-dilation condition. Then, $|\Gamma_D| = \alpi+\betn$ (since optimum matching in $\B(\bA, \bB)$  has $\alpi+\betn$ interconnections). The maximum cardinality of $\Gamma_A \cap \Gamma_D$ is the set of interconnections that can serve both the conditions, i.e., accessibility and the no-dilation. Thus, $\alpi$ is the maximum cardinality of $\Gamma_A \cap \Gamma_D$. In other words, $\alpi$ is the maximum number of interconnections present in sets $\Gamma_A$ and $\Gamma_D$ that can serve both the purposes. Hence $\betn = |\Gamma_D| - \alpi$, is the minimum number of interconnections in $\Gamma_D$ that are needed to meet the no-dilation condition solely.
\end{rem}
We refer to $\alpi$ as the {\em maximum commonality index} and $\betn$ as the {\em unique dilation index}. As the subsystems are more interconnected within themselves, the value of indices, $\alpi$ and $\betn$, decreases.
\section{Illustrative Example, Special Cases and Multi-input Case}\label{sec:illus}
In this section, we first give an illustrative example to demonstrate Algorithm~\ref{alg:similar_inter}. Then, we discuss few special cases and possible extensions.
\vspace*{-2.63 mm}
\subsection{Illustrative Example}\label{sec:illus_eg}
We demonstrate Algorithm~\ref{alg:similar_inter} through an illustrative example in Figure~\ref{fig:main}. The subsystems are $S_1, S_2, S_3, S_4$. The set $\cN = \{ \cN_1, \ldots, \cN_{11}\}$, where $\cN_1 = x^1_4$, $\cN_2 = x^1_5$, $\cN_3 = \{x^2_1, x^2_2, x^2_3\}$, $\cN_4 = x^2_4$, $\cN_5 = x^2_5$, $\cN_6 = \{x^3_1, x^3_2, x^3_3\}$, $\cN_7= x^3_4$, $\cN_8 = x^3_5$, $\cN_9 = \{x^4_1, x^4_2, x^4_3\}$, $\cN_{10} = x^4_4$, $\cN_{11} = x^4_5$. We first obtain an optimum matching $M^\*$ in $\B(\bA, \bB, \cN)$. $M^\* = \{ (x'^1_1,u_1)$, $(x'^1_2,x^1_4)$, $(x'^1_3,x^1_2)$, $(x'^1_4,\cN_1)$, $(x'^1_5,\cN_2)$, $(x'^2_1,x^2_2)$, $(x'^2_2,x^2_1)$, $(x'^2_3,\cN_3)$, $(x'^2_4,\cN_4)$, $(x'^2_5, \cN_5)$, $(x'^3_1,\cN_6)$, $(x'^3_2,x^3_1)$, $(x'^3_3,x^3_2)$, $(x'^3_4, \cN_7),$ $(x'^3_5, \cN_8)$, $(x'^4_1,x^4_2)$, $(x'^4_2,x^4_1)$, $(x'^4_3, \cN_9)$, $(x'^4_4, \cN_{10})$, $(x'^4_5, \cN_{11}) \}$. Here, $\alpi=11$ and $\betn = 0$. A matching $\widetilde{M}$ is obtained corresponding to this $M^\*$ as shown in Figure~\ref{fig:illus_1}. Here, blue coloured edges are the edges corresponding to $M'$ given in Step~\ref{step:M'} and the red edges are edges corresponding to $M''$ shown in Step~\ref{step:pm}. Thus the blue and red edges in Figure~\ref{fig:illus_1} together constitute matching $\widetilde{M}$. With respect to $\widetilde{M}$ the not accessible SCCs are $\{\cN_6, \cN_7, \ldots, \cN_{11}\}$. The node set $V_{\cN'} = \{x'^3_1, x'^3_4, x'^3_5, x'^4_3, x'^4_4, x'^4_5\}$ and the unique unmatched accessible node corresponding to $\widetilde{M}$ is the blue coloured node, $x^2_5$. 

Now we redefine $\widetilde{M}$.
In order to redefine $\widetilde{M}$, we first break the edge $(x'^3_1, x^4_3)$ from $\widetilde{M}$ and make edge $(x'^3_1, x^2_5)$ as shown in Figure~\ref{fig:illus_2}. After this, SCCs $\cN_6, \cN_9$ become accessible. Thus, $V_{\cN'} =\{x'^3_4, x'^3_5, x'^4_4, x'^4_5\} $. The unique unmatched node in this stage is $x^4_3$.
Now we further redefine $\widetilde{M}$. To do this we break the edge $(x'^3_4, x^4_4)$ and make the edge $(x'^3_4, x^4_3)$ as shown in Figure~\ref{fig:illus_3}. At the end of this stage, SCCs $\cN_7, \cN_{10}$ become accessible. Thus, $V_{\cN'} =\{ x'^3_5, x'^4_5\} $. The unique unmatched node in this stage is $x^4_4$. 
In the final step, we now redefine $\widetilde{M}$ by breaking edge $(x'^3_5,x^4_5)$ and making edge $(x'^3_5, x^4_4)$ as shown in Figure~\ref{fig:illus_4}. Finally, SCCs $\cN_8, \cN_{11}$ are also accessible now. This completes Step~\ref{step:endwhile}. In this example $\alpi = q =11$. Thus, $\E^q_\I = \phi$. Thus the minimum number of interconnections to make the composite system structurally controllable is equal to $\betn +q = 0+11 =11$ as shown by the red edges in Figure~\ref{fig:illus_4}.
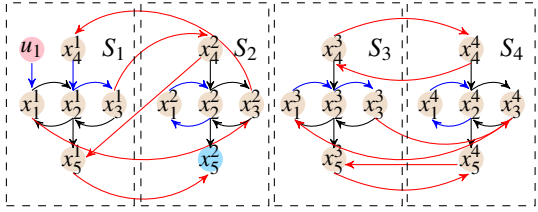
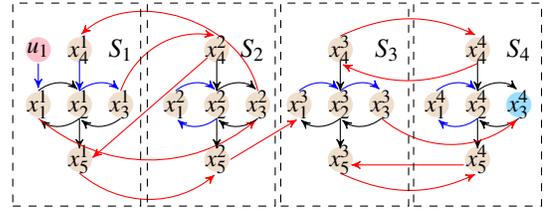
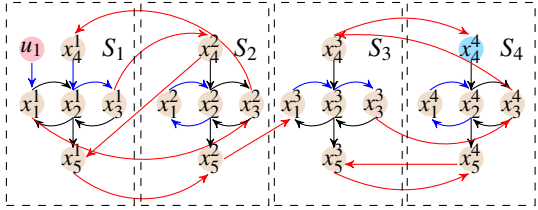
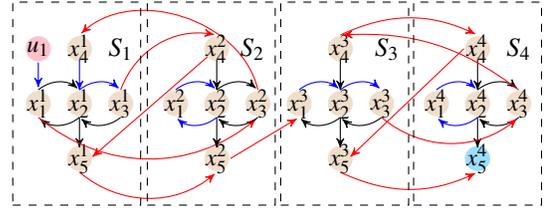
\begin{figure*}[t]
\centering
\begin{subfigure}[b]{0.45\textwidth}
\centering
\begin{tikzpicture}[scale = 0.73, ->,>=stealth',shorten >=1pt,auto,node distance=1.85cm, main node/.style={circle,draw,font=\scriptsize\bfseries}]
\definecolor{myblue}{RGB}{80,80,160}
\definecolor{almond}{rgb}{0.94, 0.87, 0.8}
\definecolor{bubblegum}{rgb}{0.99, 0.76, 0.8}
\definecolor{columbiablue}{rgb}{0.61, 0.87, 1.0}

  \fill[almond] (5.0,0) circle (6.5 pt);
  \fill[almond] (5.75,0) circle (6.5 pt);
  \fill[almond] (6.5,0) circle (6.5 pt);  
  \fill[almond] (5.75,1) circle (6.5 pt);
  \fill[almond] (5.75,-1) circle (6.5 pt);
  
  \fill[bubblegum] (5,1.0) circle (6.5 pt);
  \node at (5,1.0) {\small $u_1$};     
  \draw [blue] (5.0,0.75)  ->   (5.0,0.25); 

  \node at (5,0) {\small $x_1^1$};
  \node at (5.75,0) {\small $x_2^1$};
  \node at (6.5,0) {\small $x_3^1$};  
  \node at (5.75,1) {\small $x_4^1$};
  \node at (5.75,-1) {\small $x_5^1$};

\draw [blue] (5.75,0.8)  ->   (5.75,0.2);  
\draw (5.75,-0.25)  ->   (5.75,-0.85);  

\path[every node/.style={font=\sffamily\small}]
(5,0.25) edge[bend left = 40] node [left] {} (5.75,0.25)
(5.75,-0.25) edge[bend left = 40] node [left] {} (5,-0.25)
(5.75,0.25) edge[blue,bend left = 40] node [left] {} (6.5,0.25)
(6.5,-0.25) edge[bend left = 40] node [left] {} (5.75,-0.25)
(5,-0.25) edge[red,bend right = 40] node [right] {} (9,-0.25)
(6.5,0.25) edge[red,bend left = 40] node [left] {} (8.25,1.25)
(5.75,-1.25) edge[red,bend right = 40] node [right] {} (8.25,-1.25)
(9,0.25) edge[red,bend right = 60] node [right] {} (5.75,1.25)
(8.25,1.0) edge[red,bend left =0] node [left] {} (5.95,-1.0);

  \fill[almond] (7.5,0) circle (6.5 pt);
  \fill[almond] (8.25,0) circle (6.5 pt);
  \fill[almond] (9,0) circle (6.5 pt);  
  \fill[almond] (8.25,1) circle (6.5 pt);
  \fill[columbiablue] (8.25,-1) circle (6.5 pt);

  \node at (7.5,0) {\small $x_1^2$};
  \node at (8.25,0) {\small $x_2^2$};
  \node at (9,0) {\small $x_3^2$};  
  \node at (8.25,1) {\small $x_4^2$};
  \node at (8.25,-1) {\small $x_5^2$};

\draw (8.25,0.8)  ->   (8.25,0.2);  
\draw (8.25,-0.25)  ->   (8.25,-0.85);  

\path[every node/.style={font=\sffamily\small}]
(7.5,0.25) edge[blue,bend left = 40] node [left] {} (8.25,0.25)
(8.25,-0.25) edge[blue, bend left = 40] node [left] {} (7.5,-0.25)
(8.25,0.25) edge[bend left = 40] node [left] {} (9,0.25)
(9,-0.25) edge[bend left = 40] node [left] {} (8.25,-0.25);
  \fill[almond] (9.75,0) circle (6.5 pt);
  \fill[almond] (10.5,0) circle (6.5 pt);
  \fill[almond] (11.25,0) circle (6.5 pt);  
  \fill[almond] (10.5,1) circle (6.5 pt);
  \fill[almond] (10.5,-1) circle (6.5 pt);

  \node at (9.75,0) {\small $x_1^3$};
  \node at (10.5,0) {\small $x_2^3$};
  \node at (11.25,0) {\small $x_3^3$};  
  \node at (10.5,1) {\small $x_4^3$};
  \node at (10.5,-1) {\small $x_5^3$};
  
\draw (10.5,0.8)  ->   (10.5,0.2);  
\draw (10.5,-0.25)  ->   (10.5,-0.85);  

\path[every node/.style={font=\sffamily\small}]
(9.75,0.25) edge[blue,bend left = 40] node [left] {} (10.5,0.25)
(10.5,-0.25) edge[bend left = 40] node [left] {} (9.75,-0.25)
(10.5,0.25) edge[blue,bend left = 40] node [left] {} (11.25,0.25)
(11.25,-0.25) edge[bend left = 40] node [left] {} (10.5,-0.25);

  \fill[almond] (12.25,0) circle (6.5 pt);
  \fill[almond] (13,0) circle (6.5 pt);
  \fill[almond] (13.75,0) circle (6.5 pt);  
  \fill[almond] (13,1) circle (6.5 pt);
  \fill[almond] (13,-1) circle (6.5 pt);

  \node at (12.25,0) {\small $x_1^4$};
  \node at (13,0) {\small $x_2^4$};
  \node at (13.75,0) {\small $x_3^4$};  
  \node at (13,1) {\small $x_4^4$};
  \node at (13,-1) {\small $x_5^4$};
 
\draw (13,0.8)  ->   (13,0.2);  
\draw (13,-0.25)  ->   (13,-0.85);
\draw[red] (12.75,-1.1)  ->   (10.65,-1.1);  

\path[every node/.style={font=\sffamily\small}]
(12.25,0.25) edge[blue,bend left = 40] node [left] {} (13,0.25)
(13,-0.25) edge[blue,bend left = 40] node [left] {} (12.25,-0.25)
(13,0.25) edge[bend left = 40] node [left] {} (13.75,0.25)
(13.75,-0.25) edge[bend left = 40] node [left] {} (13,-0.25)
(11.25,-0.25) edge[red,bend right = 40] node [right] {} (13.75,-0.25)
(10.5,1.25) edge[red,bend left = 25] node [left] {} (13,1.25)
(10.5,-1.25) edge[red,bend right = 25] node [right] {} (13,-1.25)
(13.75,-0.25) edge[red,bend left = 40] node [left] {} (9.75,-0.25)
(13,0.75) edge[red,bend left = 25] node [left] {} (10.5,0.75);
\node (rect) at (5.7,0) [draw,dashed,minimum width=1.7cm,minimum height=2.7cm] {};
\node at (6.5, 1) {$S_1$};
\node (rect) at (8.15,0) [draw,dashed,minimum width=1.7cm,minimum height=2.7cm] {};
\node at (8.9, 1) {$S_2$};
\node (rect) at (10.58,0) [draw,dashed,minimum width=1.7cm,minimum height=2.7cm] {};
\node at (11.35, 1) {$S_3$};
\node (rect) at (13,0) [draw,dashed,minimum width=1.7 cm,minimum height=2.7cm] {};
\node at (13.75, 1) {$S_4$};

\end{tikzpicture}
\caption{For the matching $\widetilde{M}$ (shown in red and blue edges) SCCs $\cN_1, \cN_2, \cN_3, \cN_4, \cN_5$ are accessible. The unique accessible unmatched node with respect to this matching is $x^2_5$.}
\label{fig:illus_1}
\end{subfigure}~\hspace{8.5 mm}
\begin{subfigure}[b]{0.45\textwidth}
\centering
\begin{tikzpicture}[scale = 0.73, ->,>=stealth',shorten >=1pt,auto,node distance=1.85cm, main node/.style={circle,draw,font=\scriptsize\bfseries}]
\definecolor{myblue}{RGB}{80,80,160}
\definecolor{almond}{rgb}{0.94, 0.87, 0.8}
\definecolor{bubblegum}{rgb}{0.99, 0.76, 0.8}
\definecolor{columbiablue}{rgb}{0.61, 0.87, 1.0}

  \fill[almond] (5.0,0) circle (6.5 pt);
  \fill[almond] (5.75,0) circle (6.5 pt);
  \fill[almond] (6.5,0) circle (6.5 pt);  
  \fill[almond] (5.75,1) circle (6.5 pt);
  \fill[almond] (5.75,-1) circle (6.5 pt);

  \fill[bubblegum] (5,1.0) circle (6.5 pt);
  \node at (5,1.0) {\small $u_1$};     
  \draw [blue] (5.0,0.75)  ->   (5.0,0.25); 
  
  \node at (5,0) {\small $x_1^1$};
  \node at (5.75,0) {\small $x_2^1$};
  \node at (6.5,0) {\small $x_3^1$};  
  \node at (5.75,1) {\small $x_4^1$};
  \node at (5.75,-1) {\small $x_5^1$};

\draw [blue] (5.75,0.8)  ->   (5.75,0.2);  
\draw (5.75,-0.25)  ->   (5.75,-0.85);  

\path[every node/.style={font=\sffamily\small}]
(5,0.25) edge[bend left = 40] node [left] {} (5.75,0.25)
(5.75,-0.25) edge[bend left = 40] node [left] {} (5,-0.25)
(5.75,0.25) edge[blue,bend left = 40] node [left] {} (6.5,0.25)
(6.5,-0.25) edge[bend left = 40] node [left] {} (5.75,-0.25)
(5,-0.25) edge[red,bend right = 40] node [right] {} (9,-0.25)
(6.5,0.25) edge[red,bend left = 40] node [left] {} (8.25,1.25)
(5.75,-1.25) edge[red,bend right = 40] node [right] {} (8.25,-1.25)
(9,0.25) edge[red,bend right = 60] node [right] {} (5.75,1.25)
(8.25,1.0) edge[red,bend left =0] node [left] {} (5.95,-1.0);

  \fill[almond] (7.5,0) circle (6.5 pt);
  \fill[almond] (8.25,0) circle (6.5 pt);
  \fill[almond] (9,0) circle (6.5 pt);  
  \fill[almond] (8.25,1) circle (6.5 pt);
  \fill[almond] (8.25,-1) circle (6.5 pt);

  \node at (7.5,0) {\small $x_1^2$};
  \node at (8.25,0) {\small $x_2^2$};
  \node at (9,0) {\small $x_3^2$};  
  \node at (8.25,1) {\small $x_4^2$};
  \node at (8.25,-1) {\small $x_5^2$};

\draw (8.25,0.8)  ->   (8.25,0.2);  
\draw (8.25,-0.25)  ->   (8.25,-0.85);  

\path[every node/.style={font=\sffamily\small}]
(7.5,0.25) edge[blue,bend left = 40] node [left] {} (8.25,0.25)
(8.25,-0.25) edge[blue, bend left = 40] node [left] {} (7.5,-0.25)
(8.25,0.25) edge[bend left = 40] node [left] {} (9,0.25)
(9,-0.25) edge[bend left = 40] node [left] {} (8.25,-0.25);
  \fill[almond] (9.75,0) circle (6.5 pt);
  \fill[almond] (10.5,0) circle (6.5 pt);
  \fill[almond] (11.25,0) circle (6.5 pt);  
  \fill[almond] (10.5,1) circle (6.5 pt);
  \fill[almond] (10.5,-1) circle (6.5 pt);

  \node at (9.75,0) {\small $x_1^3$};
  \node at (10.5,0) {\small $x_2^3$};
  \node at (11.25,0) {\small $x_3^3$};  
  \node at (10.5,1) {\small $x_4^3$};
  \node at (10.5,-1) {\small $x_5^3$};
  
\draw (10.5,0.8)  ->   (10.5,0.2);  
\draw (10.5,-0.25)  ->   (10.5,-0.85);  

\path[every node/.style={font=\sffamily\small}]
(9.75,0.25) edge[blue,bend left = 40] node [left] {} (10.5,0.25)
(10.5,-0.25) edge[bend left = 40] node [left] {} (9.75,-0.25)
(10.5,0.25) edge[blue,bend left = 40] node [left] {} (11.25,0.25)
(11.25,-0.25) edge[bend left = 40] node [left] {} (10.5,-0.25);

  \fill[almond] (12.25,0) circle (6.5 pt);
  \fill[almond] (13,0) circle (6.5 pt);
  \fill[columbiablue] (13.75,0) circle (6.5 pt);  
  \fill[almond] (13,1) circle (6.5 pt);
  \fill[almond] (13,-1) circle (6.5 pt);

  \node at (12.25,0) {\small $x_1^4$};
  \node at (13,0) {\small $x_2^4$};
  \node at (13.75,0) {\small $x_3^4$};  
  \node at (13,1) {\small $x_4^4$};
  \node at (13,-1) {\small $x_5^4$};
 
\draw (13,0.8)  ->   (13,0.2);  
\draw (13,-0.25)  ->   (13,-0.85);
\draw[red] (12.75,-1.1)  ->   (10.65,-1.1);  

\path[every node/.style={font=\sffamily\small}]
(12.25,0.25) edge[blue,bend left = 40] node [left] {} (13,0.25)
(13,-0.25) edge[blue,bend left = 40] node [left] {} (12.25,-0.25)
(13,0.25) edge[bend left = 40] node [left] {} (13.75,0.25)
(13.75,-0.25) edge[bend left = 40] node [left] {} (13,-0.25)
(11.25,-0.25) edge[red,bend right = 40] node [right] {} (13.75,-0.25)
(10.5,1.25) edge[red,bend left = 25] node [left] {} (13,1.25)
(10.5,-1.25) edge[red,bend right = 25] node [right] {} (13,-1.25)
(8.5,-1) edge[red,bend left = 0] node [left] {} (9.75,-0.25)
(13,0.75) edge[red,bend left = 25] node [left] {} (10.5,0.75);
\node (rect) at (5.7,0) [draw,dashed,minimum width=1.7cm,minimum height=2.7cm] {};
\node at (6.5, 1) {$S_1$};
\node (rect) at (8.15,0) [draw,dashed,minimum width=1.7cm,minimum height=2.7cm] {};
\node at (8.9, 1) {$S_2$};
\node (rect) at (10.58,0) [draw,dashed,minimum width=1.7cm,minimum height=2.7cm] {};
\node at (11.35, 1) {$S_3$};
\node (rect) at (13,0) [draw,dashed,minimum width=1.7 cm,minimum height=2.7cm] {};
\node at (13.75, 1) {$S_4$};

\end{tikzpicture}
\caption{For the matching $\widetilde{M}$ (shown in red and blue edges) SCCs $\cN_1, \cN_2, \cN_3, \cN_4, \cN_5, \cN_6, \cN_9$ are accessible. The unique accessible unmatched node with respect to this matching is $x^4_3$.}
\label{fig:illus_2}
\end{subfigure}

\begin{subfigure}[b]{0.45\textwidth}
\centering
\begin{tikzpicture}[scale = 0.73, ->,>=stealth',shorten >=1pt,auto,node distance=1.85cm, main node/.style={circle,draw,font=\scriptsize\bfseries}]
\definecolor{myblue}{RGB}{80,80,160}
\definecolor{almond}{rgb}{0.94, 0.87, 0.8}
\definecolor{bubblegum}{rgb}{0.99, 0.76, 0.8}
\definecolor{columbiablue}{rgb}{0.61, 0.87, 1.0}

  \fill[almond] (5.0,0) circle (6.5 pt);
  \fill[almond] (5.75,0) circle (6.5 pt);
  \fill[almond] (6.5,0) circle (6.5 pt);  
  \fill[almond] (5.75,1) circle (6.5 pt);
  \fill[almond] (5.75,-1) circle (6.5 pt);

  \fill[bubblegum] (5,1.0) circle (6.5 pt);
  \node at (5,1.0) {\small $u_1$};     
  \draw [blue] (5.0,0.75)  ->   (5.0,0.25); 
  
  \node at (5,0) {\small $x_1^1$};
  \node at (5.75,0) {\small $x_2^1$};
  \node at (6.5,0) {\small $x_3^1$};  
  \node at (5.75,1) {\small $x_4^1$};
  \node at (5.75,-1) {\small $x_5^1$};

\draw [blue] (5.75,0.8)  ->   (5.75,0.2);  
\draw (5.75,-0.25)  ->   (5.75,-0.85);  

\path[every node/.style={font=\sffamily\small}]
(5,0.25) edge[bend left = 40] node [left] {} (5.75,0.25)
(5.75,-0.25) edge[bend left = 40] node [left] {} (5,-0.25)
(5.75,0.25) edge[blue,bend left = 40] node [left] {} (6.5,0.25)
(6.5,-0.25) edge[bend left = 40] node [left] {} (5.75,-0.25)
(5,-0.25) edge[red,bend right = 40] node [right] {} (9,-0.25)
(6.5,0.25) edge[red,bend left = 40] node [left] {} (8.25,1.25)
(5.75,-1.25) edge[red,bend right = 40] node [right] {} (8.25,-1.25)
(9,0.25) edge[red,bend right = 60] node [right] {} (5.75,1.25)
(8.25,1.0) edge[red,bend left =0] node [left] {} (5.95,-1.0);

  \fill[almond] (7.5,0) circle (6.5 pt);
  \fill[almond] (8.25,0) circle (6.5 pt);
  \fill[almond] (9,0) circle (6.5 pt);  
  \fill[almond] (8.25,1) circle (6.5 pt);
  \fill[almond] (8.25,-1) circle (6.5 pt);

  \node at (7.5,0) {\small $x_1^2$};
  \node at (8.25,0) {\small $x_2^2$};
  \node at (9,0) {\small $x_3^2$};  
  \node at (8.25,1) {\small $x_4^2$};
  \node at (8.25,-1) {\small $x_5^2$};

\draw (8.25,0.8)  ->   (8.25,0.2);  
\draw (8.25,-0.25)  ->   (8.25,-0.85);  

\path[every node/.style={font=\sffamily\small}]
(7.5,0.25) edge[blue,bend left = 40] node [left] {} (8.25,0.25)
(8.25,-0.25) edge[blue, bend left = 40] node [left] {} (7.5,-0.25)
(8.25,0.25) edge[bend left = 40] node [left] {} (9,0.25)
(9,-0.25) edge[bend left = 40] node [left] {} (8.25,-0.25);
  \fill[almond] (9.75,0) circle (6.5 pt);
  \fill[almond] (10.5,0) circle (6.5 pt);
  \fill[almond] (11.25,0) circle (6.5 pt);  
  \fill[almond] (10.5,1) circle (6.5 pt);
  \fill[almond] (10.5,-1) circle (6.5 pt);

  \node at (9.75,0) {\small $x_1^3$};
  \node at (10.5,0) {\small $x_2^3$};
  \node at (11.25,0) {\small $x_3^3$};  
  \node at (10.5,1) {\small $x_4^3$};
  \node at (10.5,-1) {\small $x_5^3$};
  
\draw (10.5,0.8)  ->   (10.5,0.2);  
\draw (10.5,-0.25)  ->   (10.5,-0.85);  

\path[every node/.style={font=\sffamily\small}]
(9.75,0.25) edge[blue,bend left = 40] node [left] {} (10.5,0.25)
(10.5,-0.25) edge[bend left = 40] node [left] {} (9.75,-0.25)
(10.5,0.25) edge[blue,bend left = 40] node [left] {} (11.25,0.25)
(11.25,-0.25) edge[bend left = 40] node [left] {} (10.5,-0.25);

  \fill[almond] (12.25,0) circle (6.5 pt);
  \fill[almond] (13,0) circle (6.5 pt);
  \fill[almond] (13.75,0) circle (6.5 pt);  
  \fill[columbiablue] (13,1) circle (6.5 pt);
  \fill[almond] (13,-1) circle (6.5 pt);

  \node at (12.25,0) {\small $x_1^4$};
  \node at (13,0) {\small $x_2^4$};
  \node at (13.75,0) {\small $x_3^4$};  
  \node at (13,1) {\small $x_4^4$};
  \node at (13,-1) {\small $x_5^4$};
 
\draw (13,0.8)  ->   (13,0.2);  
\draw (13,-0.25)  ->   (13,-0.85);
\draw[red] (12.75,-1.1)  ->   (10.65,-1.1);  

\path[every node/.style={font=\sffamily\small}]
(12.25,0.25) edge[blue,bend left = 40] node [left] {} (13,0.25)
(13,-0.25) edge[blue,bend left = 40] node [left] {} (12.25,-0.25)
(13,0.25) edge[bend left = 40] node [left] {} (13.75,0.25)
(13.75,-0.25) edge[bend left = 40] node [left] {} (13,-0.25)
(11.25,-0.25) edge[red,bend right = 40] node [right] {} (13.75,-0.25)
(10.5,1.25) edge[red,bend left = 25] node [left] {} (13,1.25)
(10.5,-1.25) edge[red,bend right = 25] node [right] {} (13,-1.25)
(8.5,-1) edge[red,bend left = 0] node [left] {} (9.75,-0.25)
(13.75,0.25) edge[red,bend right = 20] node [right] {} (10.5,1.25);
\node (rect) at (5.7,0) [draw,dashed,minimum width=1.7cm,minimum height=2.7cm] {};
\node at (6.5, 1) {$S_1$};
\node (rect) at (8.15,0) [draw,dashed,minimum width=1.7cm,minimum height=2.7cm] {};
\node at (8.9, 1) {$S_2$};
\node (rect) at (10.58,0) [draw,dashed,minimum width=1.7cm,minimum height=2.7cm] {};
\node at (11.35, 1) {$S_3$};
\node (rect) at (13,0) [draw,dashed,minimum width=1.7 cm,minimum height=2.7cm] {};
\node at (13.75, 1) {$S_4$};

\end{tikzpicture}
\caption{For the matching $\widetilde{M}$ (shown in red and blue edges) SCCs $\cN_1, \cN_2, \cN_3, \cN_4, \cN_5, \cN_6, \cN_7, \cN_9, \cN_{10}$ are accessible. The unique accessible unmatched node with respect to this matching is $x^4_4$.}
\label{fig:illus_3}
\end{subfigure}~\hspace{8.5 mm}
\begin{subfigure}[b]{0.45\textwidth}
\centering
\begin{tikzpicture}[scale = 0.73, ->,>=stealth',shorten >=1pt,auto,node distance=1.85cm, main node/.style={circle,draw,font=\scriptsize\bfseries}]
\definecolor{myblue}{RGB}{80,80,160}
\definecolor{almond}{rgb}{0.94, 0.87, 0.8}
\definecolor{bubblegum}{rgb}{0.99, 0.76, 0.8}
\definecolor{columbiablue}{rgb}{0.61, 0.87, 1.0}

  \fill[almond] (5.0,0) circle (6.5 pt);
  \fill[almond] (5.75,0) circle (6.5 pt);
  \fill[almond] (6.5,0) circle (6.5 pt);  
  \fill[almond] (5.75,1) circle (6.5 pt);
  \fill[almond] (5.75,-1) circle (6.5 pt);

  \fill[bubblegum] (5,1.0) circle (6.5 pt);
  \node at (5,1.0) {\small $u_1$};     
  \draw [blue] (5.0,0.75)  ->   (5.0,0.25); 
  
  \node at (5,0) {\small $x_1^1$};
  \node at (5.75,0) {\small $x_2^1$};
  \node at (6.5,0) {\small $x_3^1$};  
  \node at (5.75,1) {\small $x_4^1$};
  \node at (5.75,-1) {\small $x_5^1$};

\draw [blue] (5.75,0.8)  ->   (5.75,0.2);  
\draw (5.75,-0.25)  ->   (5.75,-0.85);  

\path[every node/.style={font=\sffamily\small}]
(5,0.25) edge[bend left = 40] node [left] {} (5.75,0.25)
(5.75,-0.25) edge[bend left = 40] node [left] {} (5,-0.25)
(5.75,0.25) edge[blue,bend left = 40] node [left] {} (6.5,0.25)
(6.5,-0.25) edge[bend left = 40] node [left] {} (5.75,-0.25)
(5,-0.25) edge[red,bend right = 40] node [right] {} (9,-0.25)
(6.5,0.25) edge[red,bend left = 40] node [left] {} (8.25,1.25)
(5.75,-1.25) edge[red,bend right = 40] node [right] {} (8.25,-1.25)
(9,0.25) edge[red,bend right = 60] node [right] {} (5.75,1.25)
(8.25,1.0) edge[red,bend left =0] node [left] {} (5.95,-1.0);

  \fill[almond] (7.5,0) circle (6.5 pt);
  \fill[almond] (8.25,0) circle (6.5 pt);
  \fill[almond] (9,0) circle (6.5 pt);  
  \fill[almond] (8.25,1) circle (6.5 pt);
  \fill[almond] (8.25,-1) circle (6.5 pt);

  \node at (7.5,0) {\small $x_1^2$};
  \node at (8.25,0) {\small $x_2^2$};
  \node at (9,0) {\small $x_3^2$};  
  \node at (8.25,1) {\small $x_4^2$};
  \node at (8.25,-1) {\small $x_5^2$};

\draw (8.25,0.8)  ->   (8.25,0.2);  
\draw (8.25,-0.25)  ->   (8.25,-0.85);  

\path[every node/.style={font=\sffamily\small}]
(7.5,0.25) edge[blue,bend left = 40] node [left] {} (8.25,0.25)
(8.25,-0.25) edge[blue, bend left = 40] node [left] {} (7.5,-0.25)
(8.25,0.25) edge[bend left = 40] node [left] {} (9,0.25)
(9,-0.25) edge[bend left = 40] node [left] {} (8.25,-0.25);
  \fill[almond] (9.75,0) circle (6.5 pt);
  \fill[almond] (10.5,0) circle (6.5 pt);
  \fill[almond] (11.25,0) circle (6.5 pt);  
  \fill[almond] (10.5,1) circle (6.5 pt);
  \fill[almond] (10.5,-1) circle (6.5 pt);

  \node at (9.75,0) {\small $x_1^3$};
  \node at (10.5,0) {\small $x_2^3$};
  \node at (11.25,0) {\small $x_3^3$};  
  \node at (10.5,1) {\small $x_4^3$};
  \node at (10.5,-1) {\small $x_5^3$};
  
\draw (10.5,0.8)  ->   (10.5,0.2);  
\draw (10.5,-0.25)  ->   (10.5,-0.85);  

\path[every node/.style={font=\sffamily\small}]
(9.75,0.25) edge[blue,bend left = 40] node [left] {} (10.5,0.25)
(10.5,-0.25) edge[bend left = 40] node [left] {} (9.75,-0.25)
(10.5,0.25) edge[blue,bend left = 40] node [left] {} (11.25,0.25)
(11.25,-0.25) edge[bend left = 40] node [left] {} (10.5,-0.25);

  \fill[almond] (12.25,0) circle (6.5 pt);
  \fill[almond] (13,0) circle (6.5 pt);
  \fill[almond] (13.75,0) circle (6.5 pt);  
  \fill[almond] (13,1) circle (6.5 pt);
  \fill[columbiablue] (13,-1) circle (6.5 pt);

  \node at (12.25,0) {\small $x_1^4$};
  \node at (13,0) {\small $x_2^4$};
  \node at (13.75,0) {\small $x_3^4$};  
  \node at (13,1) {\small $x_4^4$};
  \node at (13,-1) {\small $x_5^4$};
 
\draw (13,0.8)  ->   (13,0.2);  
\draw (13,-0.25)  ->   (13,-0.85);
\draw[red] (12.75,1.0)  ->   (10.65,-1.0);  

\path[every node/.style={font=\sffamily\small}]
(12.25,0.25) edge[blue,bend left = 40] node [left] {} (13,0.25)
(13,-0.25) edge[blue,bend left = 40] node [left] {} (12.25,-0.25)
(13,0.25) edge[bend left = 40] node [left] {} (13.75,0.25)
(13.75,-0.25) edge[bend left = 40] node [left] {} (13,-0.25)
(11.25,-0.25) edge[red,bend right = 40] node [right] {} (13.75,-0.25)
(10.5,1.25) edge[red,bend left = 25] node [left] {} (13,1.25)
(10.5,-1.25) edge[red,bend right = 25] node [right] {} (13,-1.25)
(8.5,-1) edge[red,bend left = 0] node [left] {} (9.75,-0.25)
(13.75,0.25) edge[red,bend right = 20] node [right] {} (10.5,1.25);
\node (rect) at (5.7,0) [draw,dashed,minimum width=1.7cm,minimum height=2.7cm] {};
\node at (6.5, 1) {$S_1$};
\node (rect) at (8.15,0) [draw,dashed,minimum width=1.7cm,minimum height=2.7cm] {};
\node at (8.9, 1) {$S_2$};
\node (rect) at (10.58,0) [draw,dashed,minimum width=1.7cm,minimum height=2.7cm] {};
\node at (11.35, 1) {$S_3$};
\node (rect) at (13,0) [draw,dashed,minimum width=1.7 cm,minimum height=2.7cm] {};
\node at (13.75, 1) {$S_4$};

\end{tikzpicture}
\caption{For the matching $\widetilde{M}$ (shown in red and blue edges) SCCs $\cN_1, \cN_2, \cN_3, \cN_4, \cN_5, \cN_6, \cN_7, \cN_8, \cN_9, \cN_{10}, \cN_{11}$ are accessible. The unique accessible unmatched node with respect to this matching is $x^4_5$.}
\label{fig:illus_4}
\end{subfigure}
\caption{Illustrative example demonstrating Algorithm~\ref{alg:similar_inter} on subsystems $S_1, S_2, S_3$ and $S_4$. The blue and the red edges corresponds to a matching in $\B(\bA, \bB)$. The blue edges are those edges which connects two nodes in the same subsystem and the red edges are the interconnections in the matching. The set $\cN$ consists of $11$ non-top linked SCCs.}
\label{fig:main}
\end{figure*}
\vspace*{-2 mm}
\vspace*{-2.63 mm}
\subsection{Special Cases}\label{sec:special_cases}
Now, we will focus on few special cases, where the minimum number of interconnections can be more directly obtained, and see the value of $|\I^\*|$ for these cases.
\subsubsection*{Structurally Cyclic Systems}
 The first case is when $\bA_s$ is structurally cyclic\footnote{A structured system $\bA$ is said to be {\it structurally cyclic} if its state bipartite graph $\B(\bA)$ has a perfect matching.}. There exists practically important systems, for instance {\em self damped} systems including multi-agent systems and epidemic dynamics,  that are structurally cyclic \cite{PeqSouPed:15}, \cite{ChaMes:13}. Then, $\B(\bA_s)$ has a perfect matching. So the composite system does not have dilation even without using any interconnection. Thus only the accessibility condition has to be catered. For optimum matching $M^\*$ in $\B(\bA, \bB, \cN)$, our algorithm gives $\alpi+\betn = 0$. Hence, $|\I^\*| = q$. In other words, the set of interconnections needed to solve Problem~\ref{prob:similar_int} equals the number of non-top linked SCCs that are not accessible. 

\subsubsection*{Controller Canonical Form}
Now we consider another class of systems, where $\bA_s$ is in the controller canonical form and $\bB = \*\,e_{\nT}$, where $e_{\nT}$ is the last column of the $(\nT \times \nT)$ identity matrix. For example, $\bA_s =  \left[ \begin{smallmatrix} 0 & \* & 0\\ 0 & 0 & \* \\ \* & \* & \*  \end{smallmatrix}  \right]$.
 Notice that, if $\bA_s$ is in the controller canonical form, then $\B(\bA_s)$ has a perfect matching. Thus the composite system does not have dilation even without using any interconnection edge. Thus for optimum matching $M^\*$ in $\B(\bA, \bB, \cN)$, we get $\alpi+\betn = 0$. Further, $\D_i(\bA_s)$ is irreducible for all $i \in \{1, \ldots, k\}$. Thus, $|\I^\*| = q = k-1$, since each subsystem is a non-top linked SCC and exactly one subsystem is accessible without using any interconnections. 
\subsubsection*{Subsystems that are Individually Structurally Controllable}

Here we consider subsystems that are individually structurally controllable with the given input. This means that all states are accessible in each subsystem and there is no dilation in each subsystem separately. Thus all non-top linked SCCs are input accessible even in the composite system, hence accessibility is satisfied. For the dilation condition, since the subsystems are structurally identical, one of the following has to hold. (i)~$\B(\bA_s)$ has a perfect matching, or (ii)~size of the maximum matching in $\B(\bA_s)$ is $n_s - 1$ (since with just one input all subsystems were structurally controllable). In case~(i), there is no need for any interconnection to make the composite system structurally controllable. In case~(ii), exactly $k-1$ interconnections are needed since the matching size is one less in $k-1$ subsystems (one subsystem connects to input node) and this has to be achieved through interconnections. Now, analysing this case using our algorithm, $q = 0$. Thus $|\I^\*| = \betn = k-1$, if there is no perfect matching in $\B(\bA_s)$ and $|\I^\*| = 0$ otherwise.
\vspace*{-2.63 mm}
\subsection{Multi-input Case}\label{sec:extension}
The discussions and results given in this paper extends to the multi-input case. We briefly explain the outline of the extension in this subsection. For a multi-input case, consider any optimum matching $M^\*$ obtained in Step~\ref{step:MCMM} of Algorithm~\ref{alg:similar_inter}. Then, $|M^\* \cap \E_\I| \geqslant 1$. Let $|M^\* \cap \E_\cN| = \alpi$ and $|M^\* \cap \E_\I| = \betn$. The matching $\widetilde{M}$ constructed in Step~\ref{step:Mt}, consists of $\alpi+\betn$ interconnections.  Further, there exists at least one unmatched accessible node corresponding to $\widetilde{M}$.  Thus, we can attain a matching $\hat{M}$ in $\B(\bA, \bB)$ with $|\hat{M} \cap \E_\I| = \alpi+\betn$ such that $\alpi$ SCCs are accessible using the interconnections in $\hat{M}$. Hence we can achieve accessibility of $\alpi$ non-top SCCs using the same number of interconnections as before. The remaining SCCs can be made accessible using extra interconnections as in Step~\ref{step:extra_accessibility}. Note that the proofs in this paper uses two concepts: (a)~in an optimum matching $\widetilde{M}$ there exists a node matched to some input node and (b)~there exists an unmatched accessible node in $\widetilde{M}$. Both (a)~and~(b) continue to be true for the multi-input case also. Thus the algorithm and results apply to the multi-input case. 
\vspace*{-2.0 mm}
\section{Conclusion}\label{sec:conclu}
In this paper, we studied structural controllability of an LTI composite system consisting of several subsystems. The objective is to find a minimum cardinality set of interconnections among these subsystems that the composite system is structurally controllable using a specified input matrix. The analysis is done in a structured framework by using the sparsity pattern of the system matrices. In this paper we considered subsystems with identical sparsity pattern and proposed a polynomial time algorithm for solving the optimal selection of interconnections in composite systems. Given a set of structured subsystems and input matrix, we first identified the cardinality of the minimum set of interconnections that must be established to attain structural controllability (Theorem~\ref{th:opt_value}). Then we proposed an algorithm to obtain these interconnections (Algorithm~\ref{alg:similar_inter}). This algorithm identifies a set of neighbours for each subsystem such that the composite system is structurally controllable with least possible number of interconnections (Theorem~\ref{th:opt}) and has polynomial complexity (Theorem~\ref{th:comp1}). For notational convenience and brevity, we discussed single input case in this paper. However, all the analysis carried out here directly extends to the multi-input case as discussed in Section~\ref{sec:extension}. Needless to elaborate, due to duality between controllability and observability in LTI systems all results of this paper directly follows to the observability problem .
\vspace*{-2.0 mm}
\bibliographystyle{myIEEEtran}  
\bibliography{myreferences} 
\end{document}